\documentclass[12pt]{amsart}
\usepackage{amssymb, amsmath, amsmath}
\usepackage[backref]{hyperref}
\usepackage{mathrsfs}
\usepackage[alphabetic,backrefs,lite]{amsrefs}
\usepackage{amscd}   
\usepackage{fullpage}
\usepackage[all]{xy} 
\usepackage{setspace}

\DeclareFontEncoding{OT2}{}{} 
\newcommand{\textcyr}[1]{%
 {\fontencoding{OT2}\fontfamily{wncyr}\fontseries{m}\fontshape{n}\selectfont #1}}
\newcommand{\Sha}{{\mbox{\textcyr{Sh}}}}

\usepackage[usenames,dvipsnames]{color}

\newtheorem{lemma}{Lemma}[section]
\newtheorem{theorem}[lemma]{Theorem}

\newtheorem{prop}[lemma]{Proposition}

\newtheorem{claim*}{Claim}

\newtheorem{example}[lemma]{Example}

\theoremstyle{definition}
\newtheorem{remark}[lemma]{Remark}

\newcommand{\PP}{{\mathbb P}}

\newcommand{\F}{{\mathbb F}}

\newcommand{\Z}{{\mathbb Z}}

\newcommand{\Xbar}{{\overline{X}}}

\newcommand{\Adeles}{{\mathbb A}}

\newcommand{\calO}{{\mathcal O}}

\DeclareMathOperator{\Br}{Br}
\DeclareMathOperator{\Gr}{Gr}

\DeclareMathOperator{\Jac}{Jac}

\DeclareMathOperator{\Spec}{Spec}

\newcommand{\isom}{\simeq}

\numberwithin{equation}{section}
\numberwithin{table}{section}

\newcommand{\defi}[1]{\textsf{#1}} 

\def\PP{\mathbb{P}}

\def\calO{\mathcal{O}}

\DeclareMathOperator{\Mor}{Mor}
\DeclareMathOperator{\Cov}{Cov}
\DeclareMathOperator{\ab}{ab}
\DeclareMathOperator{\isog}{isog}

\usepackage{tikz-cd}

\title{Rational points on symmetric squares of constant algebraic curves over function fields}
\author{Jennifer Berg}
\address{Department of Mathematics, Bucknell University, Lewisburg, PA 17837, USA}
\email{jsb047@bucknell.edu}
\urladdr{http://sites.google.com/view/jenberg}

\author{Jos\'e Felipe Voloch}
\address{School of Mathematics and Statistics, University of Canterbury, Christchurch 8140, New Zealand}
\email{felipe.voloch@canterbury.ac.nz}
\urladdr{https://www.math.canterbury.ac.nz/~f.voloch/}

\begin{document}
\maketitle

\begin{abstract}
We consider smooth projective curves $C/\F$ over a finite field and their symmetric squares $C^{(2)}$. For a global function field $K/\F$, we study the $K$-rational points of $C^{(2)}$. We describe the adelic points of $C^{(2)}$ surviving Frobenius descent and how the $K$-rational points fit there. Our methods also lead to an explicit bound on the number of $K$-rational points of $C^{(2)}$ satisfying an additional condition. Some of our results apply to arbitrary constant subvarieties of abelian varieties, however we produce examples which show that not all of our stronger conclusions extend.
\end{abstract}

\section{Introduction}
Let $C$ be a smooth, geometrically irreducible, proper curve over a finite field $\F$, and let $K/\F$ be a global function field. In this paper, we study the set of $K$-rational points of the symmetric square $X = C^{(2)}$ (or a constant subvariety of an abelian variety) and explore the information that can be obtained about $X(K)$ from knowledge of only certain finite coverings of $X$. 

To provide some context, we briefly recall what is known about the $K$-rational points of a nice curve $C/K$, assumed to be of genus $g \ge 2$, embedded in its Jacobian $J$, and with $\Sha(J)$ finite. In \cite{PV10}*{Theorem 4.4} the Brauer-Manin obstruction is shown to be the only obstruction to weak approximation for curves over $K$, i.e., $\overline{C(K)} = C(\Adeles_K)_{\bullet}^{\Br}$, with $C$ satisfying additional but general hypotheses which exclude the case of constant curves.

However, when $C$ is a nice constant curve (or abelian variety), the Brauer-Manin obstruction is equivalent to descent under finite abelian group schemes, and in particular $C(\Adeles_K)^{\Br} = C(\Adeles_K)^{\ab}= C(\Adeles_K)^{\isog}$ \cite{CVV18}*{\S 2}. Thus in \cite{CV19}, Creutz and Voloch study $C(K)$ via descent under pull-backs of \'etale isogenies on $\Jac(C)$ defined over $\F$ together with Frobenius descent. It is shown that $C(\Adeles_K)^{F^\infty} = C(K) \cup C(\Adeles_{K,\F})$, that is, any adelic point of $C$ unobstructed by descent under the $n^{\text{th}}$ iterate of the $\F$-Frobenius isogeny on $F \colon J \to J$ for all $n \ge 1$ is global, unless it is arbitrarily divisible by Frobenius \cite{CV19}*{Theorem 1.3}. Moreover, writing $K = \F(D)$ for a nice curve $D$ over $\F$, if the genus of $D$ is less than that of $C$, then $C(\Adeles_K)^{\Br} = C(K) = C(\F)$ \cite{CV19}*{Theorem 1.5}.

Less is known about the precise relationship between finite descent obstructions and the Brauer-Manin obstruction for higher dimensional constant subvarieties $X$ of abelian varieties, (see \S\ref{subsec: descentBM} for further remarks). Still, it is reasonable to ask what information about $X(K)$ can be deduced from descent under torsors arising as pullbacks from isogenies on the abelian variety, and specifically from Frobenius descents. Moreover, Frobenius descent is interesting in its own right since it can be viewed in analogy to the differential descent obstruction in characteristic $0$ \cite{Vol}, \cite{ARM}. 

\subsection{Results for symmetric squares of constant curves over function fields} Let $C$ and $D$ be smooth, proper curves over a finite field $\F$ and set $K = \F(D)$. We note that $C$ may be embedded into its Jacobian $J$ since, by the Hasse-Weil bounds, $C$ has a 0-cycle of degree $1$ defined over $\F$. We study the $K$-rational points on the image of the symmetric square $C^{(2)}$ in $J$, parametrizing effective divisors of degree 2 on $C$.

An adelic point on a variety $X$ surviving Frobenius descent defines a class in $H^1(K, \ker(F))$ which can be related to tangent spaces of the corresponding variety $X$, see \S\ref{subsec: gauss}. Thus the setting of symmetric powers of a curve is particularly amenable to studying the information captured by the Frobenius descent obstruction since the local geometry of $X = C^{(2)}$ is well understood. In particular, there is an explicit description of the projectivized tangent space to points of $X$ in terms of secant lines on the canonically embedded curve in $\PP^{g-1}$. This perspective allows us to give a precise description of the points of $C^{(2)}$ surviving Frobenius descent in Theorem \ref{thm: global} in terms of geometric information about $C$.

The proof of Theorem \ref{thm: global} allows us to determine an explicit bound on the number of $K$-rational points of $X$ satisfying certain geometric conditions, namely that the corresponding curves are not everywhere tangent to the ``horizontal''
direction. This is made precise in the discussion immediately preceding Theorem \ref{thm: bound}, which gives the explicit
statement. As far as the authors know, even the corresponding finiteness statement is new.

While a few questions remain, the situation for the symmetric square is fairly clear. In contrast, for arbitrary
constant surfaces (or higher dimensional varieties) contained in abelian varieties, the situation is murkier.
We give constructions in \S\ref{sec: arbitraryX} of non-global points unobstructed by Frobenius descent.

\section{Notation and Background}
Let $\F$ be a finite field of characteristic $p$ and let $D$ be a smooth, proper curve over $\F$. Set $K = \F(D)$. The places of $K$ are in bijection with the set $D^1$ of closed points of $D$. Given a closed point $v \in D^1$, let $K_v, \calO_v,$ and $\F_v$ denote the corresponding completion of $K$, ring of integers, and residue field, respectively. 

Throughout this note we shall let $X$ denote a proper geometrically integral $\F$-variety and $X_K := X \otimes_K \F$. Recall that a variety over $K$ is called \defi{constant} if it is isomorphic to the base change of a variety defined over $\F$ and is called \defi{isotrivial} if it becomes constant over some finite extension of $K$. Note that we can make the identifications \[ X_K(K) = X(K) = \Mor_\F(\Spec(K),X) = \Mor_\F(D,X), \] therefore when we consider a $K$-rational point of $X$, it will be convenient to also keep in mind the existence of a morphism $f \colon D \to X$.

\subsection{Adelic and reduced adelic points} The \defi{adele ring} of $K$ is the $K$-algebra defined by the restricted direct product $\Adeles_K := \prod_{v \in D^1} (K_v: \calO_v)$, where the product runs over the closed points of $D$. Since $X$ is proper, we may identify $X(\Adeles_K) = X_K(\Adeles_K) = \prod_{v \in D^1} X(K_v)$. 

The \defi{reduced adele ring} of $K$ is the $\F$-algebra $\Adeles_{K,\F} = \prod_{v \in D^1} \F_v$, which is an $\F$-subalgebra of $\Adeles_K$. The set $X(\Adeles_{K,\F}) := \Mor_\F(\Spec(\Adeles_{K,\F}), X)$ of reduced adelic points on $X$ is a closed subset of $X(\Adeles_K)$ which can be identified with $\prod_{v \in D^1} X(\F_v)$. When the latter is endowed with the product of the discrete topologies it coincides with the subspace topology determined by $X(\Adeles_{K,\F}) \subset X(\Adeles_K)$. Such points play an important role in the description of various descent obstructions on constant curves and are still necessary to consider for higher dimensional subvarieties of constant abelian varieties.

\subsection{Descent Obstructions} Consider the category $\Cov(X_K)$ of $X_K$ torsors under finite group schemes $G$ over $K$; for a detailed account, see \cite{Sto07}*{Section 4}. We say that an adelic point $P \in X_K(\Adeles_K) = X(\Adeles_K)$ is \defi{unobstructed by}, or \defi{survives}, the torsor $(X',G) \in \Cov(X_K)$ if it lifts to an adelic point on some twist of $(X',G)$. Equivalently, $P$ survives $(X',G)$ if the element of $\prod_{v} H^1(K_v, G)$ given by evaluating $(X',G)$ at $P$ lies in the image of $H^1(K, G)$ under the usual diagonal map in bottom row of the following diagram.

\begin{center}
\begin{tikzcd} X(K) \ar{r} \ar{d} & H^1(K,G) \ar{d} \\
X(\Adeles_K) \ar{r} &  \prod \limits_v H^1(K_v,G)
\end{tikzcd}
\end{center}

Throughout this note, we restrict our attention to subvarieties $X$ of an abelian variety $A/\F$. Furthermore, we will consider the subset of torsors in $\Cov(X_K)$ which arise as pullbacks of isogenies $\phi \colon A' \to A$ defined over $\F$. 

The Frobenius isogeny on $A$ will be of particular importance in subsequent sections. Since $A$ is defined over $\F$, recall that there exists an $\F$-Frobenius morphism $A \to A$ defined by the $n$-fold composition of $\F_p$-Frobenius morphism with itself. Since we will be interested in isogenies with target $A$, it is convenient to consider the $\F_p$-Frobenius morphism denoted $F \colon A^{(-1)} \to A^{(p^n)} \isom A$, given by raising coordinates to their $p$th powers and, Zariski locally, defining equations for $A$ are obtained from those of $A^{(-1)}$ by taking $p$th powers. 

Specifically, we will consider the torsor $(X', \ker(F))$ under the finite abelian $\F$-group scheme $G = \ker(F)$. Recall that for a separable extension $L/K$, the long exact sequence in cohomology arising from the short exact sequence \[ 0 \to \ker(F) \to A^{(-1)} \xrightarrow{F} A \to 0\] gives rise to a homomorphism $\mu: A(L)/F(A^{(-1)}(L)) \to H^1(L,\ker F)$. Thus, to say that an adelic point 
$(x_v) \in X(\Adeles_K)$ survives $F$-descent means there exists a
global element in $H^1(K,\ker F)$ such that its image in $H^1(K_v,\ker F)$
coincides with the image of $x_v$ under $\mu$. 
To discern further information about adelic points unobstructed by Frobenius descent, we make use of the following lemma, originally due to Artin and Milne, \cite{AM}, which describes $\mu$ explicitly.

\begin{lemma}{\cite{CV19}*{Lemma 3.1}} There is a functorial (on separable extensions $L/K$) injection $\mu_L \colon H^1(L,\ker F) \to \Omega_{L/\F}^{\oplus g}$ such that the induced map $A(L)/F(A^{(-1)}(L) \to \Omega_{L/\F}^{\oplus g}$ is given by $x \mapsto x^\ast(\omega_1, \dots, \omega_g)$ where $\omega_1, \dots, \omega_g$ is a basis of holomorphic differentials on $A$.
\end{lemma}

We note that if $L/K$ is a finite separable extension, then $\Omega_{L/\F}$
is a one-dimensional vector space over $L$ and we can then identify the
projective space $\PP(\Omega_{L/\F}^{\oplus g})$ with $\PP^{g-1}(L)$.

\section{Frobenius Descent Obstruction}
In this section, we consider subvarieties $X$ of constant abelian varieties $A$ containing adelic points $(x_v) \in X(\Adeles_K)$ unobstructed by Frobenius descent. When $X$ is a smooth, proper, geometrically integral constant curve over $K$ embedded in its Jacobian, Creutz and the second author prove that any adelic point on $X$ nontrivially surviving a torsor $(X', \ker(F))$ must be global, where $F$ denotes the $\F_p$-Frobenius isogeny on $\Jac(X)$ \cite{CV19}*{Lemma 3.2}. We show that this need not be true for higher dimensional subvarieties. To exhibit such counterexamples, we must first consider a modified Gauss map for subvarieties of abelian varieties.

\subsection{The Gauss map} \label{subsec: gauss} Suppose that $X$ is a smooth subvariety of an abelian variety $A$ of dimension $g$. The Frobenius descent map, denoted by $\mu$ in \cite{CV19}*{Lemma 3.1}, assigns a tangent vector to a map $D \to A$, i.e., to a point in $A(K)$. As before, to say that a point $(x_v) \in X(\Adeles_K)$ survives Frobenius descent means that there exists an
element $\xi$ of $\Omega_{K/\F}^{\oplus g}$ 
such that for each place $v$, one has $\mu_{K_v}(x_v) = \xi$ for all $v$. In this case, if $\xi \ne 0$ we let $\gamma \colon D \to \PP^{g-1}$ denote the map corresponding to the point $[\xi] \in \PP^{g-1}(K)$. 

If $X$ has finite stabilizer in $A$, Abramovich proved that the Gauss map $X \to \Gr(\dim X, \dim A)$ is finite \cite{Abr94}*{Theorem 4}. In this case, the Gauss map is the usual one, defined by sending a smooth point $x \in X$ to the point on the Grassmannian representing the tangent space of $X$ at $x$ inside of the tangent space to $A$. This result holds in arbitrary characteristic. However, in the present context, we are impelled to use a variant of the Gauss map $h$ which has target $\PP^{g-1}$ in order to meaningfully compare with the Frobenius descent map $\mu$. As a result, degeneracies of $h$ arise in positive characteristic that are not present for the usual Gauss map. These lead to interesting examples, see \S\ref{subsec: degen}.

In order to define $h$, we let $P$ be the projectivized tangent bundle of $X$, the variety with a map $P \to X$ such that the fiber at $x$ in $X$ is the set of 1-dimensional subspaces of $T_xX$. Note that $\dim P = 2 \dim X - 1$. By translating such subspaces to the tangent space to $A$ at the origin, we obtain a map $h \colon P \to \PP^{g-1} = \PP(T_0 A)$. If $A$ is the Jacobian of a non-hyperelliptic curve $C$, we identify the projective space associated to the tangent space of $A$ at the origin with the ambient $\PP^{g-1}$ in which $C$ is canonically embedded. Each $x_v \in X(K_v)$ gives a map $\Spec(O_v) \to X$,  which then lifts to a map $\Spec(O_v) \to P$. That is, from an adelic point of $X$ one obtains an adelic point on $P$, and hence an adelic point on $\PP^{g-1}$ by composing with the map $h$ above for each place $v$ of $K$.

\begin{remark} In the case when $X = C$ is a non-hyperelliptic curve with an embedding into its Jacobian, the map $h$ coincides with the canonical embedding $\phi \colon C \to \PP^{g-1}$. The existence of a point $(x_v) \in C(\Adeles_K)$ unobstructed by Frobenius descent gives that $\gamma(D) = C$. When $X$ is a surface or higher dimensional variety, we shall see that it is often less straightforward to relate $\gamma(D)$ to the image of $h$ in $\PP^{g-1}$.
\end{remark}

\subsection{Frobenius descent on the symmetric square of a curve} In the context of higher dimensional subvarieties of constant abelian varieties, a reasonable setting in which to examine the Frobenius descent obstruction is the image of symmetric powers of a curve in its Jacobian. Let $C$ be a smooth non-hyperelliptic curve of genus $g>3$, and let $X$ be the image in $J := \Jac(C)$ of $C^{(2)},$ parametrizing effective divisors of degree $2$ on $C$. 

As above, let $P$ be the projectivized tangent bundle of $X$. The Riemann-Kempf Singularity Theorem \cite{Kempf},\cite{GH94}*{\S 2.7} relates the local geometry of $X$, specifically the tangent cones at various points, to the corresponding geometry of the linear systems on the canonical curve of $C$ in $\PP^{g-1}$. 
In particular, the projectivized tangent space to $C^{(2)}$ at $P_0 + Q_0$ is identified with the linear span of $P_0$ and $Q_0$ in $\PP^{g-1}$, i.e.,\ the secant line between $P_0$ and $Q_0$ on the canonical curve. Thus, we have a map of threefolds, $h \colon P \to S \subset \PP^{g-1}$ from $P$ to the secant variety $S$ of the canonically embedded curve. Note that this is the same map $h$ as above, but with the target restricted to $S$. 

We will show below in Lemma \ref{lem: sec} that the map $h$ is birational, hence there exists an open 
set $U \subset S$ where $h$ has an inverse $h^{-1}$. The set $U$ can be interpreted as the set of points $R$ in $S$ for which there exists a \emph{unique} pair $\{Q,Q'\}$ of distinct points of $C$ with $R \in \overline{QQ'}$, the secant line between $Q$ and $Q'$. The map $h$ fails to be bijective at the locus of points where two secant lines intersect. Observe that after fixing a point $P_0 \in C$ and varying another point $Q$, the line $\overline{P_0Q}$ passes through $P_0$, hence $h^{-1}(P_0)$ is a curve. By varying $P_0$, we see that $h^{-1}(C)$ is a surface in $P$. Moreover $C \cap U = \emptyset$.

\begin{lemma} \label{lem: sec} Let $C$ be a smooth non-hyperelliptic curve of genus $g>3$.
Suppose that $P_1, P_2, P_3, P_4$ are distinct points on $C$. If 
two secant lines between disjoint pairs of the images of these points on the canonical curve in $\PP^{g-1}$ intersect
then $P_1 + P_2 +  P_3 + P_4$ is a $g_4^1$. Moreover, if $g>4$, the map $h$ above is birational and if, in addition,
$C$ has no $g_4^1$, $h$ has an inverse in $S \setminus C$.
\end{lemma}

\begin{proof}
Suppose that the lines $\overline{P_1 P_2}$ and $\overline{P_3 P_4}$ intersect at a point $R$ in $\PP^{g-1}$. Then the geometric form of Riemann Roch \cite{ACGH85}*{p.\ 12} implies that the divisor $P_1 + P_2 + P_3 + P_4$ on $C$ is a $g_4^1$. This proves the first part of the lemma.

By Martens's theorem (\cite{Martens}, \cite[Theorem IV.5.1]{ACGH85})
under our assumptions, the set of $g_4^1$'s of $C$ is at most one dimensional,
so the set $U$ above is open and dense in $S$ and $h$ is invertible on $U$, hence birational. Finally, if 
$C$ has no $g_4^1$ then $U = S \setminus C$.
\end{proof}

\begin{remark}
A more general statement that implies that $h$ is birational as in Lemma \ref{lem: sec} is  
\cite[Exercise VIII.A.2]{ACGH85}
\end{remark}

We shall give conditions on $C$ that determine when an adelic point of $X$ surviving Frobenius descent is a global point.

\begin{lemma} \label{lem: meetU} Let $C$ be a smooth non-hyperelliptic curve of genus $g$ 
and let $X$ be the image of $C^{(2)}$ in $J = \Jac(C)$. Let $(x_v) \in X(\Adeles_K)$ be an adelic point with $\mu_{K_v}(x_v) \ne 0$ for some $v$ which survives Frobenius descent, and write $\mu(x_v) = \xi$ for some $\xi \in H^1(K,\ker F)$. Let $\gamma \colon D \to \PP^{g-1}$ be the map corresponding to $[\mu_{K_v}(x_v)] = [\xi] \in \PP^{g-1}(K)$. If $\gamma(D) \cap U \ne \emptyset$, then $(x_v) \in X(K)$. 
\end{lemma} 

\begin{proof}
Since any adelic point $(x_v) \in X(\Adeles_K)$ lifts to an adelic point of $P$, by composing with the map $h \colon P \to S \subset \PP^{g-1}$, we obtain an adelic point of $S$. By assumption we have $\mu(x_v) = [\xi]$, so this implies that $\gamma(D) \subset S$. 

If $\gamma(D) \subset S$ meets $U$, then we can form the map $\pi \circ h^{-1} \circ \gamma \colon D \to X$, where $\pi \colon P \to X$ is the projection. So $\gamma$ comes from a point of $X(K)$.  Indeed, this follows from the uniqueness condition in the description of $U$ given above, which guarantees that $h^{-1}$ is well defined.
\end{proof}

If we further assume $C$ has no $g^1_3$, then the condition that $\gamma(D)$ meets $U$ is enough to conclude that the adelic point is global. Together with this observation, we give the following description of adelic points surviving Frobenius descent.

\begin{theorem} \label{thm: global} Let $C$ be a smooth curve of genus $g$ and assume that $C$ has no $g^1_2$, $g^1_3$, nor $g^1_4$. Let $Z \subset X(\Adeles_K)$ be the subset of adelic points represented by degree 2 effective divisors on $C$ of the form $y + (Q_v)$ where $y \in C(K)$ and $(Q_v) \in C(\Adeles_{K,\F})$. Then the subset of adelic points of $X$ unobstructed by Frobenius descent is \[ X(\Adeles_K)^{F^\infty} = X(K) \cup X(\Adeles_{K,\F}) \cup Z.\]
\end{theorem}

\begin{proof} Let $(x_v) \in X(\Adeles_K)^{F^\infty}$ so that in particular, it is an adelic point surviving $(X', \ker F)$. Suppose first that $\mu_{K_v}(x_v) \ne 0$ for some $v$. As before, we write $\mu(x_v) = \xi$ and let $\gamma \colon D \to \PP^{g-1}$ be the corresponding map. 

If $\gamma(D) \ne C$, then by Lemma \ref{lem: sec}, since $C$ contains no $g^1_4$, we have $\gamma(D) \cap U \ne \emptyset$. Thus by Lemma \ref{lem: meetU}, we find $(x_v) \in X(K)$. If instead $\gamma(D) = C$ so that $\gamma(D)$ is disjoint from $U$, we claim that $(x_v) \in Z$.  Since $\gamma(D) = C$, it in fact defines a point $Q$ of $C(K)$. Therefore, if we write $(x_v) = (Q_v + Q_v')$, then the observation that $\mu_{K_v}(x_v) = \mu_K(y) \ne 0$ tells us that $Q$ is contained in the secant line $\overline{Q_v Q_v'}$.
Since $C$ has no $g^1_3$, $C$ does not have any trisecant lines and hence no point of $C$ is on a secant line through two other points of $C$. Thus, we recover from $y$ a point $Q$ of $C(K)$ which occurs in $\{Q_v,Q'_v\}$
for all $v$. One possibility is that $Q_v = Q_v' = Q$ for all $v$ and the point $Q_v+Q_v'$ of $X$ is global. Let us fix $Q_v = Q$
and let $Q_v'$ be the other point, which may or may not be $Q$. Then we have an adelic point $(Q_v')$ of $C$ with a global value of
$\mu$. It is therefore either global itself (hence is $Q$), or is in the image of Frobenius (when $\mu(Q_v')= 0$) by \cite{CV19}. 
In the latter case, we can write $Q_v' = F^n(R_v)$ for some $n \ge 1$ and $(R_v) \in C(\Adeles_{K})$. 
Since we assume that $(x_v) \in X(\Adeles_K)^{F^\infty}$, we get that $(R_v) \in C(\Adeles_{K})^F$ so by \cite{CV19},
either $(R_v)$ is global and, in turn, so are $Q_v', x_v$ or $(R_v)$ is in the image of Frobenius. This shows that
either $x_v$ is global or in $Z$, by \cite{CV19}*{Theorem 1.2}.

Otherwise, $\mu_{K_v}(x_v) = 0$ for all $v$, which implies $(x_v) \in F(J^{(-1)}(\Adeles_K))$ and thus lifts to a point $(y_v) \in X^{(-1)}(\Adeles_K)$. The proof that any non-global such $(x_v)$ is in $ X(\Adeles_{K,\F})$ follows in the same manner as that of \cite{CV19}*{Theorem 1.2}, but we record it here for completeness. Since $\Sha(K, \ker(F)) = 0$, $(x_v)$ does not lift to any nontrivial twist of $(X',\ker F)$. Hence $(y_v) \in X^{(-1)}(\Adeles_K)^{F^\infty}$. Since we can iterate this argument, either we find that $(x_v)$ is global, or $(x_v) \in F^n(X(\Adeles_K))$ for all $n \ge 1$, and hence in $X(\Adeles_{K,\F})$. Thus, we have $X(\Adeles_K)^{F^\infty} \subset X(K) \cup X(\Adeles_{K,\F}) \cup Z$.

For the reverse containment, if $(x_v) \in Z$, then by definition $(x_v) = y + F(z_v)$ for some $y \in J(K)$ and $(z_v) \in F(J^{(-1)}(\Adeles_K))$. Hence the class of $(x_v)$ in $\prod_v H^1(K_v, \ker F)$ is in the image of $H^1(K,\ker F)$ and thus survives $F$-descent. Finally, 
$X(\Adeles_{K,\F}) \subset X(\Adeles_{K,\F})^{F^\infty}$.
\end{proof}

Let us call a point $y \in X(K)$ \defi{horizontal} if the curve $D \to X$ corresponding to $y$ is everywhere tangent to the 
horizontal curves $C \to X, P \mapsto P+P_0, P_0$ fixed. The following estimate is a consequence of the above proof:

\begin{theorem}
\label{thm: bound}
Let $C$ be a smooth curve of genus $g$ and assume that $C$ has no $g^1_2$, $g^1_3$, nor $g^1_4$.
The number of points of $X(K)$ which are neither horizontal nor in the image of Frobenius is bounded by
$(p^r-1)/(p-1)$ where $r$ is the rank of $J(K)$.
\end{theorem} 

\begin{proof}
The proof of Theorem \ref{thm: global} shows that a point $y$ of $X(K)$ with $\mu(y) \ne 0$ for which the corresponding map 
$\gamma \colon D \to \PP^{g-1}$ is not contained in $C$ can be recovered from $\gamma$. On the other hand,
$\gamma$ depends only on $\mu(y)$ up to scalars in the image of $J(K)/F(J^{(-1)}(K) \setminus \{0\}$. 
Otherwise, again from the proof of Theorem \ref{thm: global}, if the corresponding map 
$\gamma \colon D \to \PP^{g-1}$ is contained in $C$, giving $Q \in C(K)$ then either $y=2Q$ and can also be recovered
from $\mu(y)$ up to scalars or is horizontal. In the former case, we can also uniquely recover the point from the
class in the image of $J(K)/F(J^{(-1)}(K) \setminus \{0\}$ but the classes we obtain are distinct from those of the first part
of the argument. Since 
$J(K)/F(J^{(-1)}(K)$ is a group of order at most $p^r$, the total number of these classes is bounded by $(p^r-1)/(p-1)$.
\end{proof}

We note that there can be infinitely many horizontal points, since if $P, Q \in C(K), P+F^n_q(Q) \in X$ 
is horizontal for all $n \ge 1$, where $F_q: C \to C$ is the $\F_q$-Frobenius. Also, some hypothesis is needed in
both Theorems \ref{thm: global} and \ref{thm: bound}. If $D$ is an elliptic curve and $C \to D$ is a map of degree $2$
(i.e. $C$ is bielliptic and has a $g^1_4$), then the fibers of the map $C \to D$ define a map $D \to X$. This map can be
composed with isogenies of $D$ to give rise to infinitely many maps $D \to X$, that is, points of $X(K)$. These points
are not horizontal. We can also construct adelic points, unobstructed by Frobenius descent, by taking suitable varying local components from this infinite collection of global points. See Example \ref{ex: quartic}.

\subsection{Frobenius descent on arbitrary subvarieties} \label{sec: arbitraryX}

We now consider subvarieties $X$ of constant abelian varieties $A$ with $\dim A \ge 3$, and no longer assume that $X$ is the symmetric square of a curve. We will construct $X$ in such a way that guarantees the existence of non-global adelic points unobstructed by Frobenius descent. As before, we fix a smooth curve $D$ over $\F$, and let $K = \F(D)$. 

\begin{lemma} \label{lem: bertini} Let $C_1$ and $C_2$ be two curves contained in $A$. Then there exists an irreducible divisor $X$ of $A$ containing both $C_1$ and $C_2$.
\end{lemma}

\begin{proof} Such an $X$ exists since if $V$ is a closed subset of codimension at least $2$ in $A$ containing $C_1$ and $C_2$, then by \cite{Gun17}*{Theorem 1.1} (and independently \cite{Wutz}*{Theorem 1.2}), a finite field analogue of a Bertini-type theorem for hypersurface sections containing a given subscheme, one can find an irreducible divisor of $A$ containing $V$. 
\end{proof}

\begin{prop} \label{prop: F-unobstr} 

Assume there exists points $P_1, Q \in A(K)$ and let $P_2 := P_1 + F(Q)$, where $F$ denotes the Frobenius isogeny on $A$. Let $D_1$ and $D_2$ be the images of the corresponding maps in $\Mor_{\F}(D,A)$. Then there exists a subvariety $X$ of $A$ containing $D_1$ and $D_2$. Moreover, there exist adelic points $(P_v) \in X(\Adeles_K)$ unobstructed by Frobenius descent. 
\end{prop}

\begin{proof} 
We apply Lemma \ref{lem: bertini} to $A$ and the curves $D_1$ and $D_2$ to obtain such an $X$. To exhibit adelic points on $X_K$ unobstructed by Frobenius descent, we separate the places of $K$ into two disjoint nonempty sets $S_1 \sqcup S_2$. Define the adelic point $(P_v) \in X(\Adeles_K)$ by $P_v := P_i$ if $v \in S_i$. Let $z_v \in A(K_v)$ be such that $z_v = 0$ when $v \in S_1$ and $z_v = Q$, defined above, when $v \in S_2$. Then, for each place $v$ we have $P_v = P_1 + F(z_v)$, since if $v \in S_1$ then $P_v = P_1 + F(0)$ and if $v \in S_2$ then $P_v = P_2 = P_1 + F(Q)$. Since both $S_1$ and $S_2$ are nonempty, $(P_v) \in X(\Adeles_K)$ is not a global point. Moreover, by construction, $(P_v)$ is unobstructed by Frobenius descent.
\end{proof}

However, we claim that any such adelic point is obstructed by a finite \'etale torsor.

\begin{prop} The point $(P_v) \in X(\Adeles_K)$ from Proposition~\ref{prop: F-unobstr} is obstructed by a torsor obtained from the pullback of the multiplication-by-$n$ isogeny on $A$ for some integer $n$. 
\end{prop}

\begin{proof} Assume for the sake of contradiction that $(P_v)$ is unobstructed for each integer $n$. Then $(P_v) = R + n(y_v)$ for some $R \in A(K)$ and $(y_v) \in X(\Adeles_K)$. Since $P_v$ is either $P_1$ or $P_2$, both of which are global points in $J(K)$, this equality implies that for each $v$, the point $y_v \in A(K)$ as long as $n$ and $p = \text{char}(K)$ are coprime. This would imply that $P_1 - P_2$ is in the image of the multiplication by $n$ map for each such $n$. But the Mordell-Weil group of $A$ is finitely generated, so this cannot occur. 
\end{proof}

\begin{prop}
Assume that the map $h \colon P \to \PP^{g-1}$ is generically finite. Suppose that there exists $(x_v) \in X(\Adeles_K)$ which survives Frobenius descent and such that $\mu(x_v) \ne 0$ and let $\gamma \colon D \to \PP^{g-1}$ be the map arising from the corresponding global point on $A$. If $\gamma(D)$ is contained in the open set where $h$ is finite and $h^{-1}(\gamma(D))$ is irreducible then $h^{-1}(\gamma(D))$ gives rise to a point in $X(K)\setminus X(\F)$.
\end{prop} 

\begin{proof}
As in the statement of Lemma \ref{lem: meetU} we have $\mu(x_v) = \mu(y) \ne 0$ for some $y \in J(K)$ and $y$ gives
us a map $D \to \PP^{g-1}$ and we let $C$ denote the image of $D$ in $\PP^{g-1}$ under this map. Under the stated hypotheses, $C' := h^{-1}(C)$ is an irreducible curve in $P$. Their fiber product $D' := C' \times_C D$ is curve which covers $D$ and since $C'$ is irreducible, so too is $D'$. As in the beginning of this section, the existence of the adelic point $(x_v)$ gives us an adelic point $(\tilde{x}_v)$ of $P$ which then lands in $D'$ by our construction. It follows that $D' \to D$ is a cover of curves such that almost all primes split, thus by the Chebotarev density theorem, we must have that $D'$ is birational to $D$. So we have a map $D \to D'$ and we also have maps $D' \to C' \to P \to X$ by construction, which compose to give
a map $D \to X$, hence a point in $X(K)$. 
\end{proof}

\begin{remark} Under the assumptions that $h$ is generically finite, $(x_v) \in X(\Adeles_K)$ survives Frobenius descent, and $\gamma(D)$ is contained in the open set where $h$ is finite, we conclude that the only way $X(K)\setminus X(\F)=\emptyset$ is if $D'$ is reducible. An interesting question is whether the examples provided above are typical.
\end{remark}

\subsection{Degeneracy of the Gauss map} \label{subsec: degen} 

When $X$ is a surface contained in an abelian threefold, we find that $h \colon P \to \PP^2$, thereby mapping the three dimensional projectivized tangent bundle $P$ to the projective plane and hence cannot be finite. The following example illustrates that in this context, there exist non-global adelic points on such a surface $X$ unobstructed by Frobenius descent.

\begin{example} 
\label{ex: quartic} 
Assume $\F$ has odd characteristic. Let $C$ be the quartic curve with affine equation $x^4 + y^4 = 1$, $D$ the curve $x^4 + z^2 = 1$, and let $X$ be the image of $C^{(2)}$ in the Jacobian $J$ of $C$. Consider the map $C \to D$ given by $(x,y) \mapsto (x,y^2)$. By taking the fibers of this map, we obtain a map $f \colon D \to C^{(2)}$. Let $g \colon D \to C^{(2)}$ be the composition of the map $D \to D$ taking $(x,z) \mapsto (-x,z)$ with $f$, and let $P_f$ and $P_g$ represent the points of $C^{(2)}(K)$ corresponding to $f$ and $g$, respectively.

By \cite{CV19}*{Lemma 3.1}, the Frobenius descent map $\mu$ induced on $J$ is given by $P \mapsto P^\ast(\omega_1, \dots, \omega_g)$ where the $\omega_i$ form a basis of holomorphic differentials of $C$. In this case such a basis is given by $\{ \frac{dx}{y^3}, \,  \frac{x\,dx}{y^3}, \, \frac{dx}{y^2} \}$, and we find $\mu(P_f) = (0,0,\frac{2 dx}{z}) = \mu(P_g)$. This is independent of the characteristic, hence the difference of these two points is constant and thus in the image of Frobenius, i.e., $P_f - P_g = F(z)$ for some $z \in J(K)$. 

Following the same construction as in the proof of Proposition \ref{prop: F-unobstr}, we define an adelic point $(P_v)$ of $X$ by splitting the places of $K$ into $S_1 \sqcup S_2$ and alternating between $P_f$ for $v \in S_1$ and $P_g$ for $v \in S_2$. Let $z_v = z$ when $v \in S_1$ and $z_v = 0$ when $v \in S_2$. Then by construction $P_v$ is not global but survives Frobenius descent.   
\end{example}

As previously noted, degeneracies of $h$ arise in positive characteristic that are not present for the usual Gauss map. 
We owe the following example to Dan Abramovich. Take a surface $X_0$ in a large dimensional abelian variety $A_0$, 
such that $A_0$ has $a$-number equal to $\dim A_0$,  with $h$ having three-dimensional image on $X_0$. 
Consider the quotient $A$ of $A_0$ by a general sub-group-scheme of the Frobenius kernel with tangent space $V$ of dimension $\dim A-3$. Then the image $X$ of $X_0$ in $A$ is birational to $X_0$ 
and its corresponding map $h$ has image in $\PP(TA/V)$ which is a projective plane. We have not been able to construct a
failure of the Frobenius descent obstruction in this setting.

\subsection{Descent and the Brauer-Manin obstruction} \label{subsec: descentBM}
At present, there are several difficulties in extending the result that the Brauer Manin obstruction is equivalent to finite abelian descent for curves to arbitrary constant subvarieties of abelian varieties. The proof of \cite{PV10}*{Prop.\ 4.6} uses the fact that $\Br(\Xbar) = 0$ when $X$ is a curve; for higher dimensional $X$ it is possible for transcendental Brauer elements to exist. In the number fields case, Creutz proves that there are no transcendental Brauer-Manin obstructions on abelian varieties \cite{Creutz20}, but it is unknown whether this extends to subvarieties of abelian varieties over a global function field $K$. Additionally, if the Neron-Severi group $\text{NS}(\Xbar) \not \isom \Z$ (which occurs e.g., for $X = C^{(2)}$ as follows from the corresponding classical result for $C^2$) and moreover if $H^1(K, \text{NS}(\Xbar))$ is non-trivial, then there may be additional Brauer obstructions not coming from finite abelian descent.

\section*{Acknowledgements}

The authors are grateful to Dan Abramovich and Brendan Creutz for comments and suggestions. The second author acknowledges support from 
the Marsden Fund Council administered by the Royal Society of New Zealand.

\end{document}